\documentclass{article}
\usepackage[utf8]{inputenc}
\usepackage[english]{babel}
\usepackage{authblk}
\usepackage{setspace}
\usepackage{amsfonts}
\usepackage[margin=1.25in]{geometry}
\usepackage{graphicx}
\usepackage{lineno}
\usepackage[section]{placeins}
\usepackage{microtype}
\usepackage{caption}
\DeclareCaptionFont{ls}{\lsstyle} 
\usepackage[standard]{ntheorem}
\usepackage{mathtools}
\usepackage{empheq}
\usepackage{amssymb}
\usepackage{color}
\usepackage{pict2e}
\usepackage{xfrac}
\usepackage{bm}
\usepackage{cite}

\definecolor{purple}{rgb}{.9,0,.9}
\definecolor{green}{rgb}{0,.7,0}

\newcommand{\Nat}{\mathbb{N}}

\newcommand{\Real}{\mathbb{R}}

\numberwithin{equation}{section}

\newcommand{\qedwhite}{\hfill \ensuremath{\Box}}

\newcommand{\beqn}{\begin{equation}}
\newcommand{\eeqn}{\end{equation}}

\newcommand{\symbfnt}[1]{{\cal #1}} 

\newcommand{\cA}{{\symbfnt{A}}}
\newcommand{\cC}{{\symbfnt{C}}}\newcommand{\cD}{{\symbfnt{D}}}

\newcommand{\cI}{{\symbfnt{I}}}
\newcommand{\cK}{{\symbfnt{K}}}\newcommand{\cL}{{\symbfnt{L}}}
\newcommand{\cM}{{\symbfnt{M}}}
\newcommand{\cP}{{\symbfnt{P}}}
\newcommand{\cQ}{{\symbfnt{Q}}}\newcommand{\cR}{{\symbfnt{R}}}
\newcommand{\cS}{{\symbfnt{S}}}

\newcommand{\cX}{{\symbfnt{X}}}

\newcommand{\be}{{\bf e}}

\newcommand{\bn}{{\bf n}}

\newcommand{\bu}{{\bf u}}

\newcommand{\ve}{\varepsilon}

\newcommand{\sgn}{\operatorname{sgn}}

\newcommand{\len}{\operatorname{len}}
\newcommand{\arctantwo}{\operatorname{arctan2}}
\newcommand{\ks}{\kappa_\sigma}

\usepackage{letltxmacro}
\makeatletter
\let\oldr@@t\r@@t
\def\r@@t#1#2{%
\setbox0=\hbox{$\oldr@@t#1{#2\,}$}\dimen0=\ht0
\advance\dimen0-0.2\ht0
\setbox2=\hbox{\vrule height\ht0 depth -\dimen0}%
{\box0\lower0.4pt\box2}}
\LetLtxMacro{\oldsqrt}{\sqrt}
\renewcommand*{\sqrt}[2][\ ]{\oldsqrt[#1]{#2} }
\makeatother

\makeatletter
\newcommand*{\inlineequation}[2][]{%
  \begingroup
    \refstepcounter{equation}%
    \ifx\\#1\\%
    \else
      \label{#1}%
    \fi
    \relpenalty=10000 %
    \binoppenalty=10000 %
    \ensuremath{%
      #2%
    }%
    ~\@eqnnum
  \endgroup
}
\makeatother

\usepackage{hyperref} 
\hypersetup{
    colorlinks=true,
    linkcolor=red,
    urlcolor=cyan,
    pdftitle={The Nonlocal Curvature of Various 2D Curves},
    pdfpagemode=FullScreen,
    hypertexnames=true,
}
\usepackage[noabbrev]{cleveref}

\usepackage{subfig}

\title{Approximating the Nonlocal Curvature of Planar Curves}

\author[1]{Cole Fleming}
\author[2]{Brian Seguin}

\affil[1]{Department of Mathematics and Statistics, Loyola University Chicago, Chicago, USA}
\affil[2]{Department of Mathematics and Statistics, Loyola University Chicago, Chicago, USA}

\date{}

\begin{document}

\maketitle

\begin{abstract}

\noindent Here we establish several results on the nonlocal curvature of planar curves. First we show how to express the nonlocal curvature of a curve relative to a point in terms of the nonlocal curvatures of simpler components of that curve relative to the same point. To obtain these results, it is necessary to extend the definition of nonlocal curvature to points off of the curve. We also find a formula for the nonlocal curvature of a line segment relative to any point in the plane in terms of the incomplete beta function. These results are then used to prove an approximation theorem, which states that the nonlocal curvature of a planar curve with some H\"older regularity can be approximated by the nonlocal curvature of a linear interpolating spline associated with the curve. \\



\noindent \textit{Dedicated to Matt and Jolie Fleming.}

\end{abstract}

\section{Introduction}

Motivated by the study of models in which the Laplacian is replaced by the fractional Laplacian, Caffarelli, Roquejoffre, and Savin \cite{caffarelli2009nonlocalminimalsurfaces} introduced the fractional perimeter and studied its minimizers. A set with smooth boundary that minimizes its fractional perimeter must satisfy a pointwise condition on its boundary.  As classical minimal surfaces have to have zero mean curvature at each point, this suggests using the pointwise condition for the minimizers of the fractional perimeter to define a notion of fractional mean curvature. Since fractional mean curvature depends on the entire boundary of the set, it is commonly referred to as the nonlocal mean curvature.

To extend these definitions to surfaces that need not be the boundary of an open set, Paroni, Paolo, and Seguin \cite{rparoni} introduced the concept of fractional area. The fractional area of a smooth surface $\cM\subseteq{}\Real^n$ relative to a bounded, open set $\Omega\subseteq{}\Real^n$ for any $\sigma\in(0,1)$ is defined by
\begin{equation}\label{fractional_area_def}
	\operatorname{Area_\sigma}(\cM,\Omega)\coloneqq{}\frac{1}{2\alpha_{n-1}}\int_{\cX(\cM)}\frac{\max\{\chi_\Omega(x),\chi_\Omega(y)\}}{|x-y|^{\sigma+n}}dxdy,
\end{equation}
where $\alpha_n$ denotes the volume of the $n$-dimensional unit ball, $\cX(\cM)$ is the set of pairs of points in $\Real^n$ such that the line segment between these two points intersects $\cM$ an odd number of times, and $\chi_\Omega$ is the characteristic function of $\Omega$. The minimizers of this fractional area satisfy $H_\sigma(z)=0$, where $H_\sigma(z)$ is the nonlocal mean curvature at $z\in\cM$. The exact definition of $H_\sigma(z)$ in the $n=2$ case, the case of interest in this work, can be found in the next section.

Numerious results have been established for the nonlocal mean curvature. For instance, the nonlocal mean curvature converges to the classical mean curvature in the appropriate limit---namely,
\begin{equation}
	\lim_{\sigma\uparrow{}1}(1-\sigma)H_\sigma(z)=H(z).
\end{equation}
This was first established by Abatangelo and Valdinoci \cite{AV14} for surfaces that are the boundary of a set and for surfaces with or without boundary by Paroni, Paolo, and Seguin \cite{rparoni}. It was also shown by Cabré et al.~\cite{cabre} that for any nonempty, open, bounded set $E$ with $C^{2,a}$ boundary, for $a>\sigma$, and constant $H_\sigma$, $E$ is a ball. Further results have been established involving motion by nonlocal mean curvature flow via level set methods by Chambolle, Morini, and Ponsiglione\ \cite{Chambolle2015}. See also the work of Saez and Valdinoci \cite{evolution}. In addition, results surrounding the regularity of the minimizers of the fractional perimeter have appeared in literature, such as in \cite{CAFFARELLI2013843} by Cafarelli and Valdinoci. There are many other results as well. See, for example, the survey article by Dipierro  and Valdinoci \cite{Stickiness} and the references therein.

As the previous paragraph indicates, most preexisting literature surrounding nonlocal mean curvature is analytical rather than computational in nature. In fact, to date we have only encountered one example of a computational result. Namely, the nonlocal mean curvature of a sphere was calculated by Paroni, Podio-Guidugli, and Seguin \cite{Seguin}. This is likely due to the difficulties encountered when calculating the nonlocal curvature, as the definition involves an improper integral taken in the principal value sense.

Here we focus on the two-dimensional case, $n=2$, where a surface is a curve, and seek to find a method for approximating the value of the nonlocal (mean) curvature of any $C^{1,a}$ curve $\cC\subseteq{}\Real^2$ with $a>\sigma$. This is accomplished in part by Theorem~\ref{lemma3}, where we argue that for any such curve $\cC$ and point $z$, there exists a set $\cS$, which is the union of partial linear interpolating splines of $\cC$, such that the nonlocal curvature of $\cS$ and of $\cC$ relative to $z$ are arbitrarily close. Since a partial linear interpolating spline consists of line segments, to calculate the nonlocal curvature of such a spline one needs the nonlocal curvature of a line segment, which we obtain a formula for involving the incomplete beta function in Lemma~\ref{nlc_2p_ldef}.

To achieve the results mentioned in the previous paragraph we establish a series of results that allow one to calculate the nonlocal curvature of a curve relative to a point in terms of the nonlocal curvatures of the components that make up a decomposition of the curve relative to the same point. To accomplish this, it is necessary to extend the definition of the nonlocal curvature to a point $z$ not on the curve. When this is done, it is necessary to compute the nonlocal curvature relative to some unit vector $\bu$ that plays the role of the normal to the curve. Outlining these results, in Proposition~\ref{skpm} we decompose the curve into its parts in each half plane determined by $z$ and $\bu$, Proposition~\ref{disj_thm} involves decomposing the curve into its disjoint components, and lastly in Proposition~\ref{partit_thm} decomposes the curve into what we call radial components -- parts of the curve that intersect any ray cast from $z$ at most once.

In Section~\ref{section:general_results} we recall the definition of nonlocal curvature of a curve in $\Real^2$, as well as establish several facts about the nonlocal curvature of curves. Namely, these results aim to allow one to express the nonlocal curvature of a curve at a point through the nonlocal curvatures of its components. In Section~\ref{section:line_segment}, we calculate the nonlocal curvature of a line segment relative to any point, including points off of the line segment. In Section~\ref{section:general_evaluation} we introduce the concept of a linear interpolating spline, and argue that the nonlocal curvature of a curve may be approximated arbitrarily well by the nonlocal curvature of one such linear interpolating spline. Lastly, in Section~\ref{section:summary}, we make use of the results in Sections \ref{section:line_segment} and \ref{section:general_evaluation} to outline steps on how to obtain an approximation for the nonlocal curvature of any $C^{1,a}$ curve with $a>\sigma$.

\section{Results for general curves} \label{section:general_results}

Consider a one-dimensional $C^1$ manifold $\cC$ with boundary in $\Real^2$ | that is, $\cC$ is a simple, $C^1$ planar curve which is not necessarily connected. Assume that $\cC$ has finite length. Given an open line segment $\cL$, we say that $\cL$ and $\cC$ are transversal if at all $z\in\cL\cap\cC$, the tangent to $\cC$ at $z$ is not parallel to $\cL$. It is known that if $\cL$ and $\cC$ are transversal, then $\cC\cap\cL$ is a finite set. See, for example, Guillemin and Polack \cite{GP74}. Define $\cX(\cC)$ to be the set of all pairs of points $(x,y)\in\Real^2\times\Real^2$ such that the open line segment $\cL$ connecting $x$ and $y$ is transversal to $\cC$ and the number of elements in $\cC\cap\cL$ is odd. Let $\bn$ be a vector-valued function defined on $\cC$ such that $\bn(z)$ is normal to the curve at $z\in\cC$. With all of this in place, we can define two sets:
\begin{align}
	\nonumber \cA_e(z)&\coloneqq{}\big\{y\in \Real^2\ |\ \big((z,y)\in \cX(\cC)\ \text{and}\ (z-y)\cdot \bn(z)> 0\big)\\
	\label{cAe} &\hspace{1in}\text{or } \big((z,y)\in \cX(\cC)^c\ \text{and}\ (z-y)\cdot \bn(z)< 0\big)\big\},\\
	\nonumber \cA_i(z)&\coloneqq{}\big\{y\in \Real^2\ |\ \big((z,y)\in \cX(\cC)^c\ \text{and}\ (z-y)\cdot \bn(z)> 0\big)\\
	\label{cAi} &\hspace{1in}\text{or } \big((z,y)\in \cX(\cC)\ \text{and}\ (z-y)\cdot \bn(z)< 0\big)\big\}.
\end{align}
One can view $\cA_e(z)$ and $\cA_i(z)$ as, respectively, the exterior and interior of the curve $\cC$ at $z$ relative to $\bn(z)$. This language is motivated by the fact that in the case where $\cC$ is the boundary of an open set $\cD$ and $\bn(z)$ is its exterior unit normal at $z$, then $\cA_e(z)$ agrees with the complement of $\cD$ up to a set of zero area and $\cA_i(z)$ agrees with $\cD$ up to a set of zero area. This motivates introducing the following notation: if $A$ and $B$ are subsets of $\Real^2$, write $A\cong B$ if the area of the set $(A\setminus B)\cup (B\setminus A)$ is zero. Notice that if $f$ is an integrable function and $A\cong B$, then
\begin{equation}
	\int_A f(x)dx=\int_B f(x)dx.
\end{equation}
It is true that
\beqn
\Real^2\cong \cA_e(z)\cup\cA_i(z).
\eeqn
See Figure~\ref{fig:AeAi} for the depiction of these sets for a given curve.

\begin{figure}[h]
    \centering%
    \includegraphics[height=7cm]{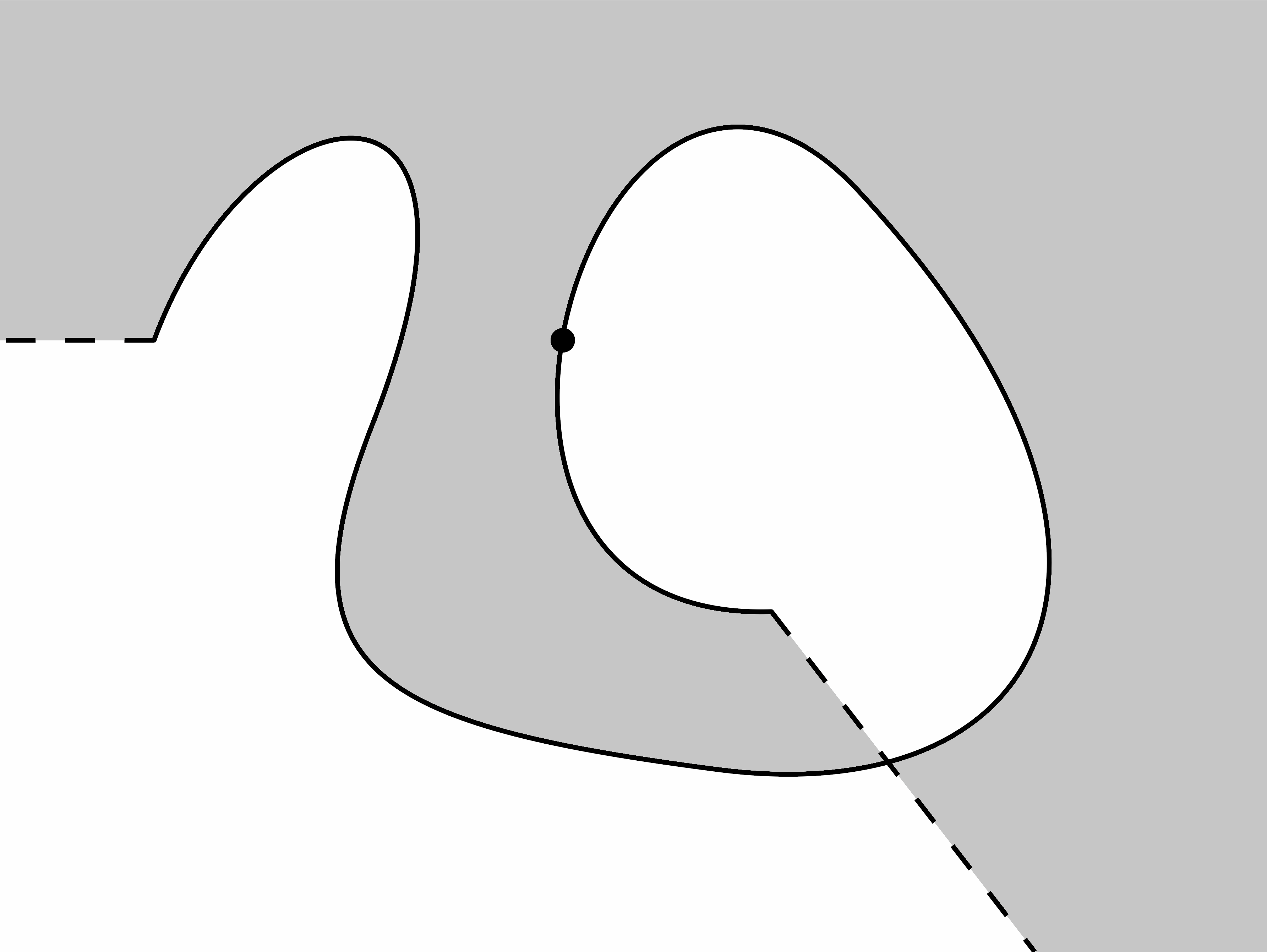}
    \put(-56,155){$\cA_i(z)$}
    \put(-100,100){$\cA_e(z)$}
    \put(-240,40){$\cA_e(z)$}
    \put(-145,127.5){$\vector(1,-0.15){25}$}
    \put(-140,134){$\bn(z)$}
    \put(-160,125){$z$}
    \caption{Depiction of the sets $\cA_e(z)$ and $\cA_i(z)$ for a given curve $\cC$ and point $z$. The grey region depicts $\cA_i(z)$ while the white region depicts $\cA_e(z)$, up to sets of zero area. The curve $\cC$ is the solid line while the dashed line segments depict other parts of the boundary between $\cA_e(z)$ and $\cA_i(z)$.}
    \label{fig:AeAi}
\end{figure}
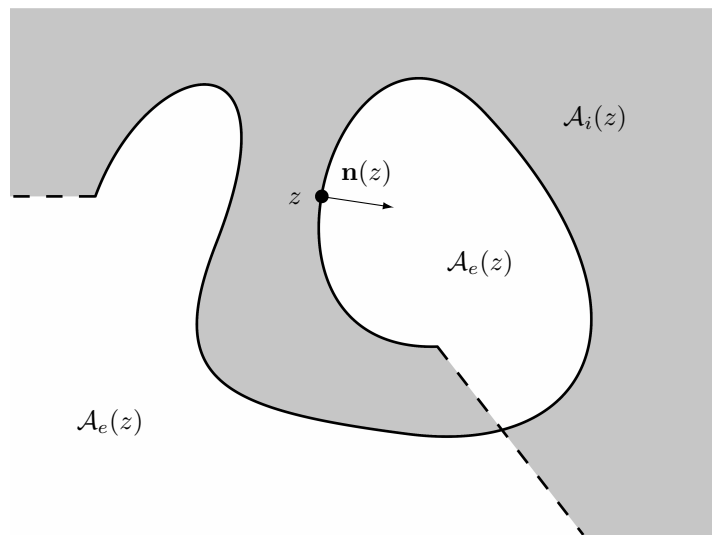

The nonlocal (mean) curvature $\ks$ at a point $z\in\cC$ is given by
\begin{equation}\label{ksdef}
    \ks(z,\cC,\bn(z))\coloneqq\frac{1}{2}\int_{\Real^2} \frac{\widehat\chi_\cC(z,y)}{|z-y|^{2+\sigma}}dy,
\end{equation}
where
\begin{equation}\label{chiS}
	\widehat\chi_\cC(z,y)\coloneqq{}
	{\small \begin{cases}
	1 & y\in\cA_i(z),\\
	0 & y\not\in \cA_i(z)\cup\cA_e(z),\\
	-1 & y\in\cA_e(z).
	\end{cases} }
\end{equation}
The integral in \eqref{ksdef} must be computed using the principle value, meaning
\begin{equation}\label{epsilon_nlc_def}
    \ks(z,\cC,\bn(z))=\lim_{\varepsilon\rightarrow 0}\frac{1}{2}\int_{\Real^2\backslash B_\varepsilon(z)} \frac{\widehat\chi_\cC(z,y)}{|z-y|^{2+\sigma}}dy,
\end{equation}
where $B_\varepsilon(z)$ is the ball of radius $\varepsilon$ centered at $z$.  Notice that the definition of $\ks(z,\cC,\bn(z))$ is unchanged if some, or all, of the end points of $\cC$ are removed. For this reason, we may consider curves that do not include their endpoints. Also, this definition makes sense when $\cC$ is only piecewise $C^1$, though $z$ must be at a point of $\cC$ where the curve is $C^1$.

According to \eqref{ksdef}, $\ks(z,\cC,\bu)$ is only defined when $z\in\cC$ and $\bu$ is one of the two unit vectors that is orthogonal to $\cC$ at $z$.  However, this definition can be extended to when $z$ is not on the curve, in which case $\bu$ can be any unit vector.  To do so, first notice that if $z\not\in\cC$, then for any unit vector $\bu$, there is an $\varepsilon>0$ such that the line segment $[z-\varepsilon\bu^\perp,z+\varepsilon\bu^\perp]$\footnote{Given $x,y\in\Real^2$, $[x,y]$ denotes the line segment with endpoints $x$ and $y$.}, with $\bu^\perp$ obtained from $\bu$ by rotating it clockwise $90^\circ$, does not intersect $\cC$.  Moreover, it follows from the definition of the nonlocal curvature that
\begin{equation}
    \ks(z,\cC\cup[z-\delta\bu^\perp,z+\delta\bu^\perp],\bu)
\end{equation}
yields the same value for any $\delta<\varepsilon$.  Thus, we can define
\begin{equation}\label{ksdefex}
    \ks(z,\cC,\bu)\coloneqq\lim_{\delta\downarrow 0}\ks(z,\cC\cup[z-\delta\bu^\perp,z+\delta\bu^\perp],\bu).
\end{equation}
From this definition, upon examination of \eqref{cAe}, \eqref{cAi}, and \eqref{ksdef}, we see that
\begin{equation}\label{inverse_unitv_identity}
    \ks(z,\cC,-\bu)=-\ks(z,\cC,\bu).
\end{equation}
Moreover, the definition \eqref{ksdefex} allows for the computation of the nonlocal curvature of the empty set:
\begin{equation}\label{nlc_emptyset}
	\ks(z,\emptyset,\bu)=\lim_{\delta\downarrow{}0}\ks(z,[z-\delta\bu^{\perp},z+\delta\bu^{\perp}],\bu)=0.
\end{equation}
For the rest of this section, consider $z\in\Real^2$ and a unit vector $\bu$ to be fixed. If $z\in\cC$, then we assume that $\bu$ is one of the normals to the curve $\cC$ at $z$.  

The first result we establish about the nonlocal curvature yields a simpler formula for $\kappa_\sigma(z,\cC,\bu)$ in the case when $\cC$ lies on one side of the line passing through $z$ with normal $\bu$.  Towards this end, we introduce the notation
\begin{equation}\label{P+-}
	\cP_- \coloneqq \{p\in\Real^2\ |\ (p-z)\cdot\bu\leq 0\}\quad \text{and}\quad   \cP_+\coloneqq \{p\in\Real^2\ |\ (p-z)\cdot\bu \geq 0\}.
\end{equation}

\begin{proposition}\label{ksrep}
If $\cC\subseteq \cP_-$, then 
\beqn\label{ksrepe}
\kappa_\sigma(z,\cC,\bu)=-\int_{\cA_e(z)\cap \cP_-}|z-y|^{-2-\sigma}dy,
\eeqn
while if $\cC\subseteq \cP_+$, then
\beqn\label{ksrepi}
\kappa_\sigma(z,\cC,\bu)=\int_{\cA_i(z)\cap \cP_+}|z-y|^{-2-\sigma}dy.
\eeqn
\end{proposition}

\begin{proof}
First assume that $\cC\subseteq \cP_-$.  Notice that in this case
\begin{equation}\label{Asetrel}
\cA_i(z)\subseteq\cP_-\quad\text{and}\quad \cA_e(z)\cong\cP_+\cup(\cA_e(z)\cap\cP_-).
\end{equation}
Thus, it follows from \eqref{ksdef} and \eqref{Asetrel} that the nonlocal curvature is given by
\begin{align}
\ks(z,\cC,\bu)&=\frac{1}{2}\Big(\int_{\cA_i(z)}-\int_{\cA_e(z)}\Big)|z-y|^{-2-\sigma}dy\\
&=\frac{1}{2}\Big(\int_{\cP_-}-\int_{\cA_e(z)\cap\cP_-}-\int_{\cA_e(z)\cap\cP_-}-\int_{\cP_+}\Big)|z-y|^{-2-\sigma}dy\\
&=-\int_{\cA_e(z)\cap \cP_-}|z-y|^{-2-\sigma}dy.
\end{align}
The proof of the other case is similar, so will be skipped.\qedwhite{}
\end{proof}

Of course, not every curve is contained in one of the half planes determined by $z$ and $\bu$.  The next result tells us that the nonlocal curvature of a curve can be decomposed into the sum of the nonlocal curvatures of the parts of the curve in each of these half planes.
\begin{proposition}\label{skpm}
    Setting $\cC_+\coloneqq \cC\cap \text{\rm Int}(\cP_+)$ and $\cC_-\coloneqq \cC\cap \text{\rm Int}(\cP_-)$, we have
    \beqn
    \kappa_\sigma(z,\cC,\bu)=\kappa_\sigma(z,\cC_+,\bu)+\kappa_\sigma(z,\cC_-,\bu).
    \eeqn
\end{proposition}

\begin{proof}
    To prove this result we will use the notation introduced in \eqref{cAe} and \eqref{cAi} but add to it the dependence on the curve.  Doing so allows us to realize that
    \beqn
    \cA_i(z,\cC)\cap\cP_+\cong\cA_i(z,\cC_+)\cap\cP_+\ \text{and}\ \cA_e(z,\cC)\cap\cP_-\cong\cA_e(z,\cC_-)\cap\cP_-.
    \eeqn
    It thus follows from the definition of the nonlocal curvature and Proposition~\ref{ksrep} that
    \begin{align*}
        \ks(z,\cC,\bu)&=\frac{1}{2}\Big(\int_{\cA_i(z,\cC)}-\int_{\cA_e(z,\cC)}\Big)|z-y|^{-2-\sigma}dy\\
        &=\frac{1}{2}\Big(\int_{\cA_i(z,\cC)\cap\cP_+}+\int_{\cA_i(z,\cC)\cap\cP_-}-\int_{\cA_e(z,\cC)\cap\cP_+}-\int_{\cA_e(z,\cC)\cap\cP_-}\Big)|z-y|^{-2-\sigma}dy\\
        &=\frac{1}{2}\Big(\int_{\cA_i(z,\cC)\cap\cP_+}+\int_{\cA_i(z,\cC)\cap\cP_-}+\int_{\cP_+}\\
        &\qquad-\int_{\cP_-}-\int_{\cA_e(z,\cC)\cap\cP_+}-\int_{\cA_e(z,\cC)\cap\cP_-}\Big)|z-y|^{-2-\sigma}dy\\
        &=\frac{1}{2}\Big(\int_{\cA_i(z,\cC)\cap\cP_+}+\int_{\cA_i(z,\cC)\cap\cP_-}+\int_{\cA_i(z,\cC)\cap\cP_+}+\int_{\cA_e(z,\cC)\cap\cP_+}\\
        &\qquad-\int_{\cA_i(z,\cC)\cap\cP_-}-\int_{\cA_e(z,\cC)\cap\cP_-}-\int_{\cA_e(z,\cC)\cap\cP_+}-\int_{\cA_e(z,\cC)\cap\cP_-}\Big)|z-y|^{-2-\sigma}dy\\
        &=\Big(\int_{\cA_i(z,\cC)\cap\cP_+}-\int_{\cA_e(z,\cC)\cap\cP_-}\Big)|z-y|^{-2-\sigma}dy\\
        &=\Big(\int_{\cA_i(z,\cC_+)\cap\cP_+}-\int_{\cA_e(z,\cC_-)\cap\cP_-}\Big)|z-y|^{-2-\sigma}dy\\
        &=\ks(z,\cC_+,\bu)+\ks(z,\cC_-,\bu).
    \end{align*}
    \qedwhite{}
\end{proof}

The nonlocal curvature is generally not an additive function of the curve. However, under a particular assumption on the curve it is. We say that a curve $\cC$ is radial with respect to the point $z$ if for any unit vector $\be$ the set $\cC\cap\{z+\delta\be\ |\ \delta>0\}$ is either empty or a singleton.

\begin{proposition}\label{disj_thm}
    If $\cC$ is radial with respect to $z$, then for any disjoint decomposition $\cC=\bigcup_{k\in\Nat}\cC_k$ of $\cC$, where each $\cC_k$ is a curve, we have
    \begin{equation}\label{nonlockadd}
        \ks(z,\cC,\bu)=\sum_{k\in\Nat}\ks(z,\cC_k,\bu).
    \end{equation}
\end{proposition}

\begin{proof}
    From Proposition~\ref{skpm} it suffices to establish this result when the curve is contained in one of the half spaces determined by $z$ and $\bu$. Thus, without loss of generality, assume that $\cC\subseteq \cP_-$. Now, fix a disjoint decomposition of $\cC=\bigcup_{k\in\Nat}\cC_k$. Since $\cC$ is radial, for each angle $\theta\in[0,2\pi)$ the ray emanating from $z$ with angle $\theta$, measured counterclockwise from the direction of the positive $x$-axis, either does not cross $\cC$ or crosses it exactly once.  Let $I$ denote the set of angles for which the ray crosses $\cC$ exactly once.  For each $\theta\in I$, we can find a unique $\hat r(\theta)\geq 0$ such that
    \begin{equation}
        z+(\hat r(\theta)\cos\theta,\hat r(\theta)\sin\theta)\in\cC.
    \end{equation}
    The function $\hat r$ describes the curve $\cC$ in polar coordinates relative to $z$.  Thus,
    \begin{equation}
        \cA_e(z)\cap\cP_-\cong\{z+(r\cos\theta,r\sin\theta)\in\Real^2\ |\ \theta\in I, r\geq\Hat{r}(\theta)\}
    \end{equation}
    and, so, from Proposition~\ref{ksrep} the nonlocal curvature is given by
    \begin{align}
        \ks(z,\cC,\bu)&=-\int_{\cA_e(z)\cap \cP_-}|z-y|^{-2-\sigma}dy\\
        &=-\int_I\int_{\hat r(\theta)}^\infty r^{-1-\sigma} dr d\theta.\label{skappa1rep}
    \end{align}
    For every $k\in\Nat$, let $I_k$ be the subset of $I$ associated with the curve $\cC_k$, meaning that
    \begin{equation}
        \cC_k=\{z+(\hat r(\theta)\cos\theta,\hat r(\theta)\sin\theta)\ |\ \theta\in I_k\},\qquad{}k\in\Nat.
    \end{equation}
    Using \eqref{skappa1rep} and the fact that $I=\bigcup_{k\in\Nat}I_k$ is a disjoint union, we conclude that
	\begin{equation}
        \ks(z,\cC,\bu)=-\sum_{k\in\Nat}\int_{I_k}\int_{\hat r(\theta)}^\infty r^{-1-\sigma} dr d\theta\\
        =\sum_{k\in\Nat}\ks(z,\cC_k,\bu),\label{result-}
	\end{equation}
establishing the desired result.    \qedwhite{}
\end{proof}

The last result of this section shows how to compute the nonlocal curvature of a curve that is not radial with respect to $z$ in terms of disjoint pieces of the curve that are radial with respect to $z$.

\begin{proposition}\label{partit_thm}
    For every $n\in\Nat$, define the sets 
    \beqn\label{Cn}
    \cC_n\coloneqq{}\{p\in\cC\mid{}|\{\lambda p+(1-\lambda)z\ |\ \lambda\in(0,1]\}\cap\cC|=n\}. 
    \eeqn
    It follows that
    \begin{equation}\label{partit_thm_statement}
        \ks(z,\cC,\bu)=\sum_{n\in\Nat}(-1)^{n+1}\kappa_\sigma(z,\cC_n,\bu)
    \end{equation}
    and
    \begin{equation}\label{fin_intersection_approx}
    	\ks(z,\cC,\bu)=\lim_{N\to\infty}\ks\big(z,{\textstyle \bigcup\limits_{n=1}^N}\cC_n,\bu\big).
	\end{equation}
\end{proposition}
\begin{proof}
    From Proposition~\ref{skpm}, it suffices to prove \eqref{partit_thm_statement} and \eqref{fin_intersection_approx} when $\cC$ is contained in one of the half planes determined by $z$ and $\bu$, so, without loss of generality, assume $\cC\subseteq{}\cP_+$.
    Considering any $N\in\Nat$, let $\widetilde{\cC}_N=\bigcup_{n=1}^{N}\cC_n$, and write
    \begin{equation}\label{template_eq}
    	\ks(z,\cC,\bu)=\ks(z,\cC,\bu)-\ks(z,\widetilde{\cC}_N,\bu)+\ks(z,\widetilde{\cC}_N,\bu).
    \end{equation}
    Following from \eqref{ksdef}, we have
    \begin{equation}
    	|\ks(z,\cC,\bu)-\ks(z,\widetilde{\cC}_N,\bu)|=\Big|\frac{1}{2}\int_{\Real^2} \frac{\widehat\chi_\cC(z,y)}{|z-y|^{2+\sigma}}dy-\frac{1}{2}\int_{\Real^2} \frac{\widehat\chi_{\widetilde{\cC}_N}(z,y)}{|z-y|^{2+\sigma}}dy\Big|.
    \end{equation}
    For every $n\in\Nat$, define the sets\footnote{Given $x,y\in\Real^2$, $(x,y]$ denotes the line segment from $x$ to $y$ including the end point $y$ but not $x$.}
    \begin{align}
    	\cK_n &\coloneqq{} \{p\in\Real^2\mid{}|(z,p]\cap{}\widetilde{\cC}_N|\geq{}n\},\label{step1_proof_2_1}\\
    	\cR_n &\coloneqq{} \cK_n\setminus{}\cK_{n+1},\label{step1_proof_2}
    \end{align}
    We then notice that
    \beqn
    \widehat\chi_\cC(z,y)=\widehat\chi_{\widetilde{\cC}_N}(z,y)\quad{}\text{for}\ y\in\Real^2\setminus{}\cK_N
    \eeqn
	and, hence,
	\begin{equation}\label{kN_bound_dif}
		|\kappa_\sigma(z,\cC,\bu)-\ks(z,\widetilde{\cC}_N,\bu)|\leq{}\int_{\cK_N}|z-y|^{-\sigma-2}dy.
	\end{equation}
	Now, for the moment, assume that $\cC_N$ is nonempty. Let $d_N$ be the smallest distance from $z$ to $\cC_N$ and $[\theta_N,\theta_{N+1}]$ be the smallest interval such that
	\begin{equation}
		\cC_N\subseteq{}\{z+(r\cos\theta,r\sin\theta)\ |\ r>0,\ \theta\in[\theta_N,\theta_{N+1}]\}.
	\end{equation}
	Then,
	\begin{align}
		\int_{\cK_N}|z-y|^{-\sigma-2}dy &= \int_{\theta_N}^{\theta_{N+1}}\int_{d_N}^{\infty}r^{-\sigma-1}drd\theta\\
		&= (\theta_{N+1}-\theta_N)\tfrac{1}{\sigma}d_N^{-\sigma}\label{kN_Bound}.
	\end{align}
	Now, consider the arc $\cA_N$ of the circle with radius $d_N$ centered at $z$ between the angles $\theta_N$ and $\theta_{N+1}$, which has length
	\begin{equation}\label{aN_Bound}
		\len(\cA_N)=d_N(\theta_{N+1}-\theta_N).
	\end{equation}
	By construction, $\cA_N$ and $\cC_N$ both lie in the wedge with vertex $z$ determined by the angles $\theta_N$ and $\theta_{N+1}$. Since $d_N$ is the minimal distance from $z$ to $\cC_N$, $\len(\cA_N)\leq{}\len(\cC_N)$. Therefore, applying this inequality, \eqref{kN_Bound}, and \eqref{aN_Bound} to \eqref{kN_bound_dif} we have
	\begin{equation}\label{dif_almost_fin_bound}
		|\ks(z,\cC,\bu)-\ks(z,\widetilde{\cC}_N,\bu)|\leq{}\tfrac{1}{\sigma}d_N^{-\sigma-1}\len(\cC_N).
	\end{equation}
	Next, since
	\begin{equation}\label{cup_C_eq}
		\len(\cC)=\len\Big(\bigcup_{k\in\Nat}\cC_k\Big)=\sum_{k\in\Nat}\len(\cC_k)
	\end{equation}
	and $\cC$ has finite length, the summation in \eqref{cup_C_eq} must converge and, so, $\len(\cC_k)$ must approach 0 as $k$ approaches infinity. Moreover, it similarly follows from the finite length of $\cC$ that $d_N$ must approach a nonzero number as $N$ approaches infinity. Therefore, taking the limit of both sides of \eqref{dif_almost_fin_bound}, we have
	\begin{equation}\label{lim_dif_0_2}
		\lim_{N\to\infty}|\ks(z,\cC,\bu)-\ks(z,\widetilde{\cC}_N,\bu)|\leq{}\tfrac{1}{\sigma}\lim_{N\to\infty}d_N^{-\sigma-1}\len(\cC_N)=0.
	\end{equation}
	Now, observe that even if $\cC_N$ were to be empty, the left-hand side of \eqref{kN_bound_dif} would be zero and, so, the limit on the left-hand side of \eqref{lim_dif_0_2} would still be equal to zero in this case. Moreover, observe that following from the definition of $\widetilde{\cC}_N$ and \eqref{lim_dif_0_2}, we obtain \eqref{fin_intersection_approx}.	
	
	By \eqref{ksrepi} we have for any natural number $n\in[1,N]$ that
    \begin{equation}
        \kappa_\sigma(z,\cC_n,\bu) = \int_{\cK_n}|z-y|^{-\sigma-2}dy.
    \end{equation}
    Therefore
    \begin{equation}\label{summation_nat_kn_lbl}
        \sum_{n=1}^{N}(-1)^{n+1}\ks(z,\cC_n,\bu) = \sum_{n=1}^{N}(-1)^{n+1}\int_{\cK_n}|z-y|^{-\sigma-2}dy.
    \end{equation}
    Notice by \eqref{step1_proof_2_1} that $\cK_{N+1}=\emptyset$. Thus, we can split the rightmost summation in \eqref{summation_nat_kn_lbl} into the sum over even and odd $n$ to obtain
    \begin{align}
        \sum_{n=1}^{N}(-1)^{n+1}\ks(z,\cC_n,\bu) &= \sum_{n=1}^{\lceil{}N/2\rceil{}}\Big(\int_{\cK_{2n-1}}-\int_{\cK_{2n}}\Big)|z-y|^{-\sigma-2}dy\\
        \label{r2n_eq}&= \sum_{n=1}^{\lceil{}N/2\rceil{}}\int_{\cR_{2n-1}}|z-y|^{-\sigma-2}dy,
    \end{align}
    where $\lceil{}t\rceil{}$ denotes the ceiling of $t\in{}\Real$. On the other hand, by the definition of nonlocal curvature and \eqref{ksrepi}, we have
    \begin{equation}\label{r2n_gen}
        \ks(z,\widetilde{\cC}_N,\bu)=\sum_{n=1}^{\lceil{}N/2\rceil{}}\int_{\cR_{2n-1}}|z-y|^{-\sigma-2}dy.
    \end{equation}
    Combining \eqref{r2n_eq} and \eqref{r2n_gen}, we obtain
    \begin{equation}\label{fin_sum_2}
    	\ks(z,\widetilde{\cC}_N,\bu)=\sum_{n=1}^{N}(-1)^{n+1}\ks(z,\cC_n,\bu).
    \end{equation}
    Then, taking the limit of \eqref{template_eq}, we may apply both \eqref{lim_dif_0_2} and \eqref{fin_sum_2} to conclude that
    \begin{align}
		\ks(z,\cC,\bu) &=\sum_{n\in\Nat}(-1)^{n+1}\ks(z,\cC_n,\bu),
	\end{align}
which is \eqref{partit_thm_statement}.\qedwhite{}

\section{The nonlocal curvature of a line segment}\label{section:line_segment}

In this section, we aim to compute the nonlocal curvature of a line segment at any point in the plane relative to any unit vector. For ease of presentation, we introduce for $z\in[0,1]$ the notation
\begin{equation}\label{f_upsilon_def}
    \Psi_\sigma(z)\coloneqq{}\tfrac{1}{2}\sgn(z)\beta_{z^2}\big(\tfrac{1}{2},\tfrac{1+\sigma}{2}\big) ,
\end{equation}
where
\begin{equation}\label{beta_function_def}
	\beta_z(a,b)\coloneqq\int_0^{z}u^{a-1}(1-u)^{b-1}du
\end{equation}
is the incomplete beta function.

\begin{lemma}\label{lemma1}
    For $\sigma\in(0,1)$, the following are true:
    \begin{align}
	\label{lemma1.1}\int_{\frac{\pi}{2}}^{z}\sin^\sigma(\theta)d\theta &= -\Psi_\sigma(\cos z),\quad z\in[0,\pi],\\
       \label{lemma1.2}\int_{0}^z\cos^{\sigma}(\theta)d\theta &= \Psi_\sigma(\sin z),\quad z\in [-\pi/2,\pi/2].
    \end{align}
    
    \begin{proof}
	First, fix $z\in[0,\pi]$ and consider any $\theta$ strictly between $\pi/2$ and $z$. Notice that since $\sin(\theta)$ is positive and $\cos(\theta)$ and $\cos(z)$ have the same sign, we have that
        \beqn
        \int_{\frac{\pi}{2}}^{z}\sin^\sigma(\theta)d\theta=\sgn(\cos z)\int_{\frac{\pi}{2}}^{z}\frac{\sin(\theta)\sgn(\cos\theta)}{\left|\sin^{1-\sigma}(\theta)\right|}d\theta.
	\eeqn
	Furthermore, upon recalling that $|a|=\sqrt{a^2}$ for any $a\in\Real$ and $\sgn(a)=a/|a|$ for $a\ne{}0$, we obtain
	\begin{equation}\label{lemma1eqn2}
		\int_{\frac{\pi}{2}}^{z}\sin^\sigma(\theta)d\theta=\frac{\cos(z)}{2|\cos(z)|}\int_{\frac{\pi}{2}}^{z}\frac{2\sin(\theta)\cos(\theta)}{\sqrt{\cos^2(\theta)\sin^{2-2\sigma}(\theta)}}d\theta.
	\end{equation}
	Applying the substitution $u\mapsto\cos^2(\theta)$ to the right-hand side of \eqref{lemma1eqn2} yields
	\beqn
	\int_{\frac{\pi}{2}}^{z}\sin^\sigma(\theta)d\theta=-\frac{1}{2}\left(\frac{\cos(z)}{\left|\cos(z)\right|} \right )\int_0^{\cos^2(z)}u^{-\frac{1}{2}}(1-u)^{\frac{\sigma-1}{2}}du.
	\eeqn
	Using \eqref{beta_function_def}, we then have
	\begin{equation}\label{prevlemmastep}
		\int_{\frac{\pi}{2}}^{z}\sin^\sigma(\theta)d\theta=-\tfrac{1}{2}\big(\tfrac{\cos(z)}{\left|\cos(z)\right|}\big)\beta_{\cos^2(z)}\big( \tfrac{1}{2},\tfrac{1+\sigma}{2}\big).
	\end{equation}
	We may then simplify \eqref{prevlemmastep} using the notation in \eqref{f_upsilon_def} to obtain
	\beqn
	\int_{\frac{\pi}{2}}^{z}\sin^\sigma(\theta)d\theta=-\Psi_\sigma(\cos z).
	\eeqn
	One can then observe that using the substitution $\theta\mapsto{}\tfrac{\pi}{2}-\theta$ followed by $z\mapsto{}\tfrac{\pi}{2}-z$ in \eqref{lemma1.1}, we obtain \eqref{lemma1.2}.     
        \qedwhite{}
    \end{proof}
\end{lemma}

\begin{figure}[h]
    \centering%
    \includegraphics[height=6cm]{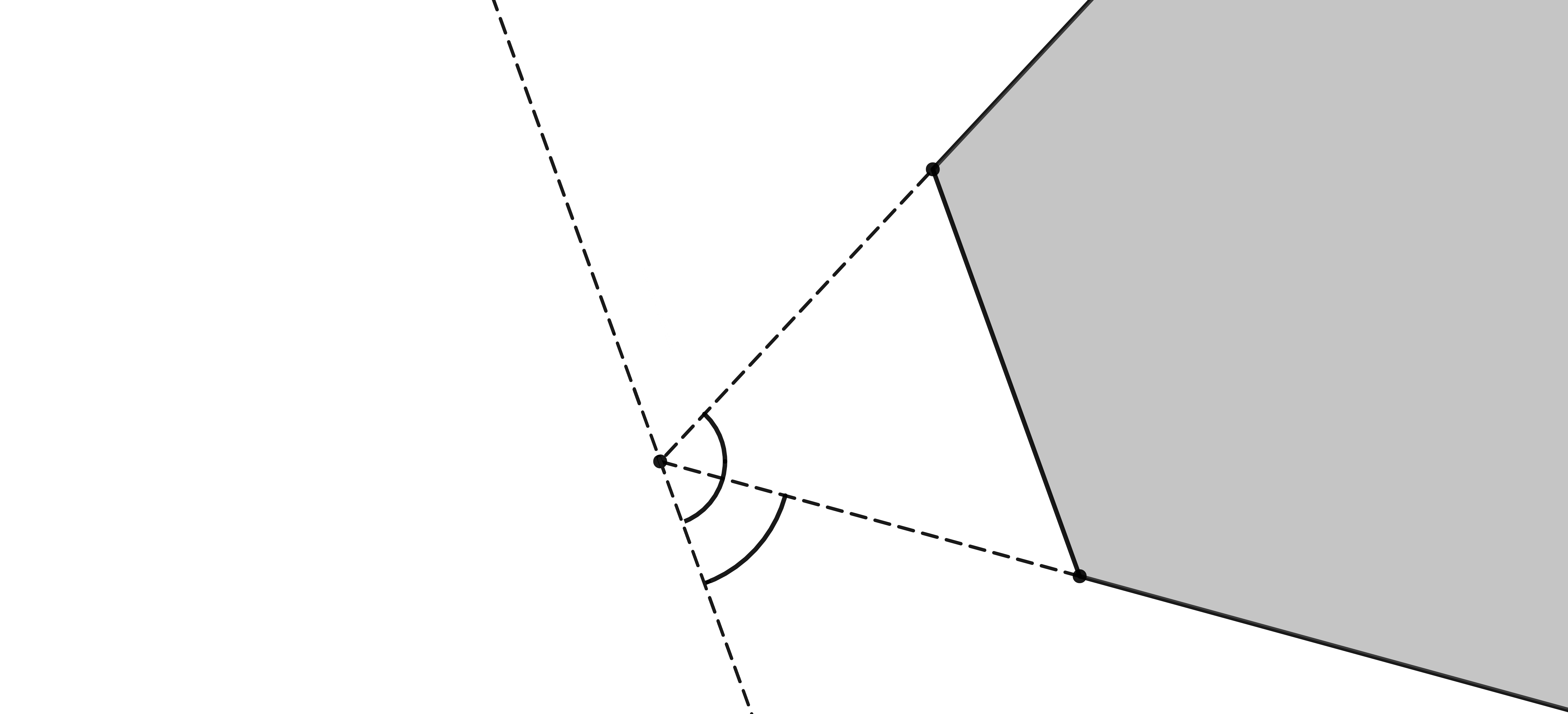}
    \put(-227,58){$z$}
    \put(-196,61){$\theta_2$}
    \put(-190,34){$\theta_1$}
    \put(-170,55){$r_1$}
    \put(-187,110){$r_2$}
    \put(-151,90){$\cL$}
    \put(-96,95){$\cA_i(z)\cap{}\cP_+$}
    \put(-223.5,78){$\vector(1,0.361173814898){25}$}
    \put(-220,85){$\bu$}
    \caption{Depiction of the line segment $\cL$ along with $z$ and $\bu$ in the plane, and the region $\cA_i(z)\cap{}\cP_+$.}
    \label{fig:3}
\end{figure}

Given a line segment $\cL$, we denote by $\cL_e$ its extension to an infinite line. Notice for any line segment $\cL$, unit vector $\bu$, and $z\in{}\cL_e$ that $\ks(z,\cC,\bu)=0$ from the definition of the nonlocal curvature. The formula for $\ks(z,\cC,\bu)$ in the case where $z\not\in\cL_e$ is presented in the next result.

\begin{lemma}\label{nlc_2p_ldef}
Consider a line segment $\cL$, a point $z\in\Real^2\setminus{}\cL_e$, and a unit vector $\bu$.  Assume that $\cL\subseteq\cP_+$(see \eqref{P+-}$_2$).  Choose a coordinate system so that $z$ is at the origin and $\bu$ points along the positive $y$-axis.  Let $(r_i,\theta_i)$, $i=1,2$, denote the polar coordinates of the end points of $\cL$ such that $\theta_1\leq{}\theta_2$, and let $(r_m,\theta_m)$ be the polar coordinates of the midpoint of $\cL$. If $\theta_\perp\in(-\pi/2,3\pi/2)$ is the angle of the ray emanating from $z$ which is perpendicular to $\cL_e$, then 
\begin{equation}\label{prop1_statement}
    \ks(z,\cL,\bu)=\frac{\sec^\sigma(\theta_m-\theta_{\perp})}{\sigma{}r_m^\sigma}\big(\Psi_\sigma(\sin(\theta_2-\theta_{\perp}))-\Psi_\sigma(\sin(\theta_1-\theta_{\perp}))\big).
\end{equation}
\end{lemma}

\begin{proof}
For the moment, assume that $\cL$ is not vertical. Let $m$ denote the slope of $\cL$ and $(x_\circ,y_\circ)=(r_\circ\cos\theta_\circ,r_\circ\sin\theta_\circ)$ be any point on $\cL_e$.  The line segment $\cL$ lies on the line
\begin{equation}\label{line_cartes}
    y-y_\circ=m(x-x_\circ)
\end{equation}
and, so, in polar coordinates it is described by the function
\begin{equation}\label{unsimp_rhat}
    \hat r(\theta)=\frac{r_\circ(\sin\theta_\circ-m\cos\theta_\circ)}{\sin\theta-m\cos\theta},\qquad \theta\in[\theta_1,\theta_2].
\end{equation}
Notice that this function is strictly positive since $z\not\in\cL_e$.  Define
\begin{equation}\label{phidef}
    \phi \coloneqq{} { \small\begin{cases}
        \arctan(1/m), & m\ne{}0, \\
        \pi/2, & m=0, \\
    \end{cases} }
\end{equation}
and observe that by the harmonic addition theorem \cite{Nay},
\begin{equation}\label{lspol}
    \hat r(\theta)\coloneqq{}\frac{r_\circ\cos(\theta_\circ+\phi)}{\cos(\theta+\phi)},\qquad \theta\in[\theta_1,\theta_2].
\end{equation}
Notice that in the case when $\cL$ is a vertical line segment, this polar function still describes $\cL$ if we set $\phi=0$. Now, utilizing \eqref{ksrepi}, we have
\begin{align}
    \kappa_\sigma(z,\cL,\bu) &= \int_{\cA_i(z,\cL)\cap{}\cP_+}|y|^{-\sigma-2}dy\\
    &= \int_{\theta_1}^{\theta_2}\int_{\hat{r}(\theta)}^{\infty}r^{-\sigma-1}drd\theta.\label{prestep_kappa}
\end{align}

Consider the case when $m=0$ and, hence, $\cL_e$ does not intersect the $x$-axis. From \eqref{phidef} we have that $\phi=\tfrac{\pi}{2}$. It then follows from \eqref{lspol} that $\hat{r}(\theta)=r_\circ\sin(\theta_\circ)\csc(\theta)$ and, so, applying this to \eqref{prestep_kappa} we find
\begin{equation}
	\ks(z,\cL,\bu)=\frac{1}{\sigma{}r_\circ^\sigma}\int_{\theta_1}^{\theta_2}\Big(\frac{\sin\theta}{\sin\theta_\circ}\Big)^{\sigma}d\theta\label{laststep_kappa_m0}.
\end{equation}
Furthermore, since $z$ is not on $\cL_e$ and $\cL\subseteq\cP_+$, we know $0\leq\theta_1<\theta_2<\pi$. From this we find that $\sin(\theta)>0$ and $\sin(\theta_\circ)>0$ and, hence, we are free to distribute the power of $\sigma$ in \eqref{laststep_kappa_m0}.  In doing so, we find
\begin{equation}
    \kappa_\sigma(z,\cL,\bu)=\frac{\csc^\sigma(\theta_\circ)}{\sigma{}r_\circ^\sigma}\int_{\theta_1}^{\theta_2}\sin^\sigma(\theta)d\theta.
\end{equation}
Lastly, we apply Lemma~\ref{lemma1} and use the notation in \eqref{f_upsilon_def} to conclude that
\begin{equation}\label{m0case}
    \kappa_\sigma(z,\cL,\bu)=\frac{\csc^\sigma(\theta_\circ)}{\sigma{}r_\circ^\sigma}(\Psi_{\sigma}(\cos\theta_1)-\Psi_{\sigma}(\cos\theta_2)).
\end{equation}
Furthermore, notice that \eqref{m0case} agrees with \eqref{prop1_statement} upon taking $r_\circ=r_m$ and noticing that $\theta_\perp=\tfrac{\pi}{2}$ in this case.

For the remainder of the proof, assume $m$ is nonzero. One may then observe that because $0\leq{}\theta_1\leq{}\theta_2\leq{}\pi$ and $\left|\phi\right|<\tfrac{\pi}{2}$, we have
\begin{equation}\label{TBOUNDS}
    -\tfrac{\pi}{2}<\theta+\phi<\tfrac{3\pi}{2},\qquad\theta\in[\theta_1,\theta_2].
\end{equation}
It follows from $\eqref{lspol}$ and the fact that $r_\circ>0$, that
\begin{align}
    &\sgn(\cos(\theta_\circ+\phi))=\sgn(\cos(\theta+\phi)),\qquad\theta\in[\theta_1,\theta_2].\label{sign_rel_thetaphi}
\end{align}

\noindent{}We then proceed with the evaluation of \eqref{prestep_kappa} to find that
\begin{equation}
    \ks(z,\cL,\bu) =\frac{1}{\sigma r_{\circ}^{\sigma}}\int_{\theta_1}^{\theta_2}\Big(\frac{\cos(\theta+\phi)}{\cos(\theta_\circ+\phi)}\Big)^{\sigma}d\theta.\label{apos_final}
\end{equation}
To proceed with the calculation, we consider two different cases.  Since $m\not=0$, $\cL_e$ intersects the $x$-axis. Denote this point of intersection by $(a,0)$.  The two cases are based on the sign of $a$.

First assume that $a>0$ and, for the moment, choose the polar point $(r_\circ,\theta_\circ)$ to represent $(a,0)$. In doing so we see that $\theta_\circ=0$. Since $|\phi|<\tfrac{\pi}{2}$, it follows that $\cos(\theta_\circ+\phi)=\cos(\phi)>0$. Hence, from \eqref{sign_rel_thetaphi}, we know that $\cos(\theta+\phi)>0$ for $\theta\in[\theta_1,\theta_2]$.  Thus, returning to $(r_\circ,\theta_\circ)$ being an arbitrary polar point on the extension of $\cL$, it follows that
\begin{equation}\label{eq50}
    \int_{\theta_1}^{\theta_2}\Big(\frac{\cos(\theta+\phi)}{\cos(\theta_\circ+\phi)}\Big)^{\sigma}d\theta=\int_{\theta_1}^{\theta_2}\frac{\cos^\sigma(\theta+\phi)}{\cos^\sigma(\theta_\circ+\phi)}d\theta.
\end{equation}
Moreover, $\cos(\theta+\phi)>0$ together with \eqref{TBOUNDS} implies that $|\theta+\phi|<\tfrac{\pi}{2}$ for $\theta\in[\theta_1,\theta_2]$. Thus, we can make use of Lemma~\ref{lemma1} to find that
\begin{align}
    \int_{\theta_1}^{\theta_2}\cos^\sigma(\theta+\phi)d\theta &= \int_{\theta_1+\phi}^{\theta_2+\phi}\cos^\sigma(\theta)d\theta\\ 
    &= \Big(\int_0^{\theta_2+\phi}-\int_0^{\theta_1+\phi}\Big)\cos^\sigma(\theta)d\theta\\
    &=\Psi_\sigma(\sin(\theta_2+\phi))-\Psi_\sigma(\sin(\theta_1+\phi)).\label{ageq0def}
\end{align}
Since $-\tfrac{1}{m}$ is the slope of the line perpendicular to $\cL_e$, we know that 
\begin{equation}\label{thetam_def_ageq0}
    \tan\theta_{\perp}=-\tfrac{1}{m}=\tan(-\phi).
\end{equation}
As $\cL_e$ intersects the positive $x$-axis, we can deduce that $|\theta_\perp|<\tfrac{\pi}{2}$.  Moreover, because $\phi$ is also in this range, it follows from \eqref{thetam_def_ageq0} that
\begin{equation}\label{thetapdefposa}
    \theta_\perp=-\phi.
\end{equation}
Putting together \eqref{apos_final}, \eqref{eq50}, \eqref{ageq0def}, and \eqref{thetapdefposa} and choosing $(r_\circ,\theta_\circ)$ to represent the midpoint of $\cL$, we obtain \eqref{prop1_statement}.

Now consider the case $a<0$. As before, temporarily choose the polar point $(r_\circ,\theta_\circ)$ to represent $(a,0)$.  Consequently, $\theta_\circ=\pi$ and, hence, $\cos(\theta_\circ+\phi)=-\cos(\phi)<0$ since $|\phi|<\tfrac{\pi}{2}$. It then follows from \eqref{sign_rel_thetaphi} that $\cos(\theta+\phi)<0$, which together with \eqref{TBOUNDS} yields $|\theta+\phi-\pi|<\tfrac{\pi}{2}$ for $\theta\in[\theta_1,\theta_2]$. Putting these results together with the fact that $\cos(x)=-\cos(x-\pi)$ for any $x$, we conclude that

\begin{equation}\label{septerm_aneg}
    \int_{\theta_1}^{\theta_2}\Big(\frac{\cos(\theta+\phi)}{\cos(\theta_\circ+\phi)}\Big)^{\sigma}d\theta=\int_{\theta_1}^{\theta_2}\frac{\cos^\sigma(\theta+\phi-\pi)}{\cos^\sigma(\theta_\circ+\phi-\pi)}d\theta.
\end{equation}
Thus, from \eqref{apos_final}, we have
\begin{equation}\label{prealeq0}
    \kappa_\sigma(z,\cL,\bu)=\frac{\sec^\sigma(\theta_\circ+\phi-\pi)}{\sigma r_\circ^\sigma}\int_{\theta_1+\phi-\pi}^{\theta_2+\phi-\pi}\cos^\sigma(\theta)d\theta.
\end{equation}
We now make use of Lemma~\ref{lemma1} to obtain
\begin{equation}\label{aneg_alm}
    \int_{\theta_1+\phi-\pi}^{\theta_2+\phi-\pi}\cos^\sigma(\theta)d\theta = \Psi_\sigma(\sin(\theta_2+\phi-\pi))-\Psi_\sigma(\sin(\theta_1+\phi-\pi)).
\end{equation}
Since the extension of $\cL$ intersects the negative $x$-axis, we deduce that $|\theta_\perp-\pi|<\tfrac{\pi}{2}$. We then notice that $|\phi|<\tfrac{\pi}{2}$, and since \eqref{thetam_def_ageq0} holds in the current case as well, we find that
\begin{equation}
    \theta_\perp=\pi-\phi.
\end{equation}
Using this result in \eqref{aneg_alm} and combining it with \eqref{prealeq0} results in \eqref{prop1_statement}.
\qedwhite{}
\end{proof}

\begin{remark}
It is possible to obtain a formula for $\theta_{\perp}$ in terms of the the coordinates of the midpoint $(x_m,y_m)$ and slope $m$ of $\cL$. Using Cartesian coordinates, we find that the ray with angle $\theta_\perp$ is given by the equation
\begin{equation}\label{crperp_cart}
    y=x\tan\theta_\perp=-\tfrac{x}{m}
\end{equation}
and the line $\cL_e$ is given by
\begin{equation}\label{midp_l_formula}
	y-y_m=m(x-x_m).
\end{equation}
Letting $(x_\circ,y_\circ)$ denote the intersection of theses two lines, one finds that
\begin{equation}\label{xprimeyprime}
    (x_0,y_0)=\Big(\frac{m(mx_m-y_m)}{m^2+1},\frac{y_m-mx_m}{m^2+1}\Big).
\end{equation}

\noindent{}Let $\arctantwo$ be the two-argument inverse tangent function, so that $\arctantwo(x,y)$ gives the angle in $(-\pi,\pi]$ associated with the point $(x,y)\in\Real^2$. Then, since $(x,y)\mapsto{}(y,-x)$ corresponds to a clockwise rotation of $\pi/2$ radians,
\begin{equation}
	\tfrac{\pi}{2}+\arctantwo(y,-x)
\end{equation}
yields the angle in $(-\pi/2,3\pi/2]$ assosciated with $(x,y)$. Therefore, since $\theta_\perp\in(-\pi/2,3\pi/2)$ is the angle associated with $(x_0,y_0)$, we have
\begin{equation}
    \theta_\perp=\tfrac{\pi}{2}+\arctantwo(y_0,-x_0).
\end{equation}
Then, with use of \eqref{xprimeyprime} and the fact that $m^2+1>0$, we conclude that
\begin{equation}\label{thetam_final}
    \theta_\perp=\tfrac{\pi}{2}+\arctantwo(y_m-mx_m, m(y_m-mx_m)).
\end{equation}

\end{remark}

\end{proof}

\section{An approximation of nonlocal curvature using splines}\label{section:general_evaluation}
Recall that a linear interpolating spline of a connected curve $\cC$ with a parameterization $\gamma:[\alpha,\beta]\to{}\Real^2$ and uniform partition $P=\{\alpha+\tfrac{\beta-\alpha}{n}k\ |\ k=0,1,\cdots{},n\}$ is given by
\begin{equation}
	\cI_n(\cC)=\bigcup_{k=1}^{n}\Big[\gamma\big(\alpha+\tfrac{\beta-\alpha}{n}(k-1)\big),\gamma\big(\alpha+\tfrac{\beta-\alpha}{n}k\big)\Big].
\end{equation}
For $\cC$ potentially not connected, a partial linear interpolating spline of $\cC$ is a linear interpolating spline $\cS$ for some connected subcurve of $\cC$.

\begin{theorem}\label{lemma3}
    Given a $C^{1,a}$ curve $\cC$ with $a>\sigma$, point $z$, unit vector $\bu$, and $\ve>0$, there are a countable number of partial linear interpolating splines $\{\cS_i\ |\ i\in\Nat\}$ of $\cC$ based on disjoint subcurves $\cC_i\subseteq\cC$, $i\in\Nat$, such that
    \begin{equation}\label{lemma3state}
        |\ks(z,\cC,\bu)-\ks(z,\cup_{i\in\Nat}\cS_i,\bu)|<\ve.
    \end{equation}
\end{theorem}
\begin{proof}
	The proof begins by establishing the result when $\cC$ satisfies a number of conditions. In subsequent steps, these assumptions are removed by using Propositions~\ref{ksrep}--\ref{partit_thm}. 
	\hypertarget{step1thm1}{}	
	
	\noindent{}\textbf{(Part 1)}
	Suppose $z\notin{}\cC$, $\cC$ is radial relative to $z$, $\cC\subseteq\cP_+$, and that $\cC$ is connected. Find a function $\hat r:[\alpha,\beta]\to{}(0,\infty)$ such that
	\begin{equation}
		\cC=\{z+(\hat r(\theta)\cos\theta,\hat r(\theta)\sin\theta)\ |\ \theta\in[\alpha,\beta]\}.
	\end{equation}
	It then follows from the definition of $\cA_i(z)$ that
	\begin{equation}
		\cA_i(z)\cap\cP_+\cong\{z+(r\cos\theta,r\sin\theta)\ |\ r\geq{}\hat r(\theta),\ \theta\in[\alpha,\beta]\}.
	\end{equation}
	Using this, we observe that \eqref{ksrepi} implies that
	\begin{align}
		\kappa_\sigma(z,\cC,\bu) &= \int_{\alpha}^{\beta}\int_{\hat r(\theta)}^{\infty}r^{-\sigma-1}drd\theta\\
		&= \frac{1}{\sigma}\int_{\alpha}^{\beta}\hat r(\theta)^{-\sigma}d\theta\label{ksc}.
	\end{align}
	Fix a natural number $n$, and construct the uniform partition $\alpha=\theta_0<\theta_1<\cdots{}<\theta_n=\beta$ of width $\delta_n=(\beta-\alpha)/n$. Let $\cI_n(\cC)$ be the linear interpolating spline of $\cC$ formed by this partition, so that
	\begin{equation}
		\cI_n(\cC)=\bigcup_{k=1}^{n}\Big[\big(\hat{r}(\theta_{k-1})\cos\theta_{k-1},\hat{r}(\theta_{k-1})\sin\theta_{k-1}\big),\big(\hat{r}(\theta_k)\cos\theta_k,\hat{r}(\theta_k)\sin\theta_k\big)\Big],
	\end{equation}
	and find $\hat r_n:[\alpha,\beta]\to{}(0,\infty)$ so that
	\begin{equation}
		\cI_n(\cC)=\{z+(\hat r_n(\theta)\cos\theta,\hat r_n(\theta)\sin\theta)\ |\ \theta\in[\alpha,\beta]\}.
	\end{equation}
	Observe that $\hat r_n^{-\sigma}(\theta_i)=\hat r^{-\sigma}(\theta_i)$ for $i=0,1,\dots{},n$. Similar to the previous argumentation that yielded \eqref{ksc}, we have
	\begin{equation}\label{kscn}
		\kappa_\sigma(z,\cI_n(\cC),\bu)=\frac{1}{\sigma}\int_{\alpha}^{\beta}\hat r_n^{-\sigma}(\theta)d\theta.
	\end{equation}
	Using \eqref{ksc} and \eqref{kscn}, we find from the triangle inequality that
	\begin{align}
		|\ks(z,\cC,\bu)-\ks(z,\cI_n(\cC),\bu)| &= \frac{1}{\sigma}\Big|\int_{\alpha}^{\beta}\big(\hat r^{-\sigma}(\theta)-\hat r^{-\sigma}(\theta_k)\big)d\theta-\int_{\alpha}^{\beta}\big(\hat r_n^{-\sigma}(\theta)-\hat r_n^{-\sigma}(\theta_k)\big)d\theta\Big|\\
		&\leq{} \frac{1}{\sigma}\sum_{k=1}^{n}\int_{\theta_{k-1}}^{\theta_k}\big(|\hat r^{-\sigma}(\theta)-\hat r^{-\sigma}(\theta_k)|+|\hat r_n^{-\sigma}(\theta)-\hat r_n^{-\sigma}(\theta_k)|\big)d\theta\label{sum_diff_eq}.
	\end{align}
	Since $\hat r^{-\sigma}$ is uniformly continuous, there is a $\delta>0$ such that for any $\theta,\phi\in[\alpha,\beta]$ with $|\theta-\phi|<\delta$ we have
	\begin{equation}\label{rem2_s3}
	|\hat{r}^{-\sigma}(\theta)-\hat{r}^{-\sigma}(\phi)|<\tfrac{\ve\sigma}{3(\beta-\alpha)}.
\end{equation}
Now, set
	\begin{equation}
		M\coloneqq{}\tfrac{1}{2}\min_{\alpha\leq{}\varphi\leq{}\beta}\hat{r}(\varphi).
	\end{equation}
Following from the fact that $\hat{r}_n$ converges uniformly to $\hat{r}$, see Ahlberg, Nilson, and Walsh \cite{Ahlberg}, we may choose $n$ large enough such that
\begin{equation}\label{min_rphi_lbl}
	0<M<\min_{\alpha\leq{}\varphi\leq{}\beta}\hat{r}_n(\varphi)
\end{equation}
and
\begin{equation}\label{rem2_s4}
	|\hat{r}_{n}(\psi)-\hat{r}(\psi)|<\tfrac{\ve\sigma}{6(\beta-\alpha)}M^{\sigma+1},\qquad{}\psi\in[\alpha,\beta].
\end{equation}
Using the function $h(x)=x^{-\sigma}$ for $x\geq{}M$ and the mean-value theorem, we know that for any $\varphi\in[\alpha,\beta]$, there is a $c_\varphi$ between $\hat{r}(\varphi)$ and $\hat{r}_n(\varphi)$ such that
\begin{equation}
	|\hat{r}_{n}^{-\sigma}(\varphi)-\hat{r}^{-\sigma}(\varphi)| \leq{} |h'(c_\varphi)||\hat{r}_{n}(\varphi)-\hat{r}(\varphi)|.
\end{equation}
Then, we observe that
\begin{equation}
	|h'(c_\varphi)|=\sigma|c_\varphi^{-\sigma-1}|\leq{}M^{-\sigma-1}
\end{equation}
and, so,
\begin{equation}\label{boundphi_var}
	|\hat{r}_{n}^{-\sigma}(\varphi)-\hat{r}^{-\sigma}(\varphi)|\leq{}M^{-\sigma-1}|\hat{r}_{n}(\varphi)-\hat{r}(\varphi)|.
\end{equation}
We then notice that by the triangle inequality, 
\begin{equation}\label{rem2_s1}
	|\hat{r}_{n}^{-\sigma}(\theta)-\hat{r}_{n}^{-\sigma}(\phi)|\leq{}
	|\hat{r}_{n}^{-\sigma}(\theta)-\hat{r}^{-\sigma}(\theta)|
	+|\hat{r}^{-\sigma}(\theta)-\hat{r}^{-\sigma}(\phi)|
	+|\hat{r}^{-\sigma}(\phi)-\hat{r}_{n}^{-\sigma}(\phi)|.
\end{equation}
Applying \eqref{boundphi_var} to \eqref{rem2_s1}, we arrive at
\begin{equation}\label{rem2_s2}
	|\hat{r}_{n}^{-\sigma}(\theta)-\hat{r}_{n}^{-\sigma}(\phi)|\leq{}|\hat{r}^{-\sigma}(\theta)-\hat{r}^{-\sigma}(\phi)|+M^{-\sigma-1}\big(|\hat{r}_{n}(\theta)-\hat{r}(\theta)|+|\hat{r}_{n}(\phi)-\hat{r}(\phi)|\big).
\end{equation}
Now, using both \eqref{rem2_s3} and \eqref{rem2_s4} in \eqref{rem2_s2} yields
\begin{equation}\label{rn_rn_bnd}
	|\hat{r}_{n}^{-\sigma}(\theta)-\hat{r}_{n}^{-\sigma}(\phi)|\leq{}\tfrac{2\ve\sigma}{3(\beta-\alpha)}\qquad \text{when}\ |\theta-\phi|<\delta.
\end{equation}
We then choose a potentially larger $n$ so that the width of our partition $\delta_n<\delta$. Making use of \eqref{rem2_s3} and \eqref{rn_rn_bnd}, we find that for any natural number $k\in[1,n]$	\begin{equation}
		|\hat r^{-\sigma}(\theta)-\hat r^{-\sigma}(\theta_k)|+|\hat r_n^{-\sigma}(\theta)-\hat r_n^{-\sigma}(\theta_k)|<\tfrac{\ve\sigma}{\beta-\alpha},\qquad{}\theta\in[\theta_{k-1},\theta_k].
	\end{equation}
	Applying this previous inequality to \eqref{sum_diff_eq}, we conclude that
	\begin{equation}
		|\ks(z,\cC,\bu)-\ks(z,\cI_n(\cC),\bu)|<\frac{1}{\sigma}\sum_{k=1}^{n}\int_{\theta_{k-1}}^{\theta_k}\frac{\ve\sigma}{\beta-\alpha}d\theta=\ve.
	\end{equation}
	\vspace{0.1in}
	
	\hypertarget{step2thm1}{}
	\noindent{}\textbf{(Part 2)} Now suppose $z\in{}\cC$, $\cC$ is connected, $\cC$ is radial relative to $z$, and $\cC\subseteq{}\cP_+$. Consider an arbitrary $\rho>0$, and let $\cC_\rho=\cC\cap{}B_\rho(z)$. Notice that for sufficiently small $\rho$, $\cC\setminus{}\cC_\rho$ is comprised of two disjoint curves, say $\cC_1$ and $\cC_2$. Observe $z\notin{}\cC_1$ and $z\notin{}\cC_2$. This allows us to apply Part~\hyperlink{step1thm1}{1} to obtain linear interpolating splines $\cI_1$ and $\cI_2$ of $\cC_1$ and $\cC_2$, respectively, such that
	\begin{equation}\label{lt_eps_ov_3}
    	|\ks(z,\cC_1,\bu)-\ks(z,\cI_1,\bu)|<\tfrac{\ve}{3}\quad{}\text{and}\quad{}|\ks(z,\cC_2,\bu)-\ks(z,\cI_2,\bu)|<\tfrac{\ve}{3}.
    \end{equation}
    By Proposition~\ref{disj_thm} we know that
    \begin{equation}
    	\ks(z,\cC,\bu)-\ks(z,\cC\setminus{}\cC_\rho,\bu)=\ks(z,\cC_\rho,\bu).
    \end{equation}
    Using this equation, the triangle inequality, and Proposition~\ref{disj_thm} again, we then find that
    \begin{align}
    	|\ks(z,\cC,\bu)-\ks(z,\cI_1\cup{}\cI_2,\bu)| &\leq{} |\ks(z,\cC,\bu)-\ks(z,\cC\setminus{}\cC_\rho,\bu)|\nonumber{}\\
		&\hspace{0.162in}+ |\ks(z,\cC\setminus{}\cC_\rho,\bu)-\ks(z,\cI_1\cup{}\cI_2,\bu)|\\
		&\leq{}|\ks(z,\cC_\rho,\bu)|+|\ks(z,\cC_1,\bu)-\ks(z,\cI_1,\bu)|\nonumber{}\\
		&\hspace{0.162in}+ |\ks(z,\cC_2,\bu)-\ks(z,\cI_2,\bu)|\label{finalbound_split_lr}.
    \end{align}
	
    We now examine the first term in \eqref{finalbound_split_lr}. After applying a change of variables so that $\bu$ is aligned with the positive $y$-axis, one can find $x_L\leq{}0\leq{}x_R$ and a $C^{1,a}$ function $f_\rho:[x_L,x_R]\to{}\Real$ such that
    \begin{equation}\label{c_rho_def}
    	\cC_\rho=\{(x,f_\rho(x))\ |\ x_L\leq{}x\leq{}x_R\}.
	\end{equation}
    Since $z\in \cC$, we must have $\bu$ orthogonal to $\cC$ at $z$. It follows that $f_\rho(0)=f_\rho'(0)=0$. Moreover, since we are assuming that $\cC\subseteq\cP_+$, we have $f_\rho(x)\geq{}0$ for all $x\in[x_L,x_R]$.

    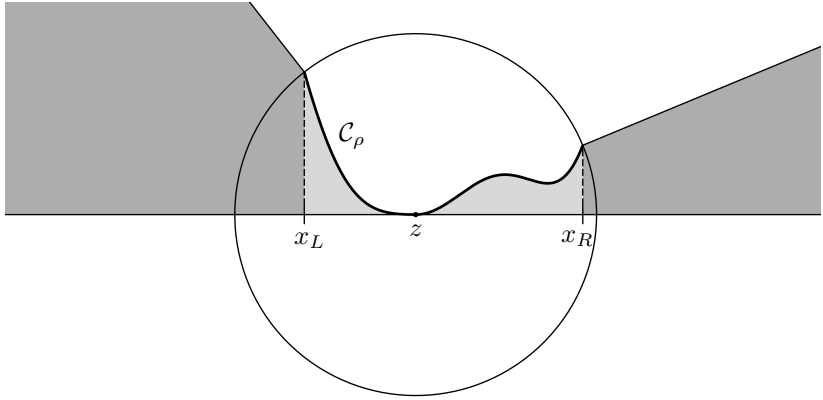
\begin{figure}[h]
    	\centering%
    	\includegraphics[height=6cm]{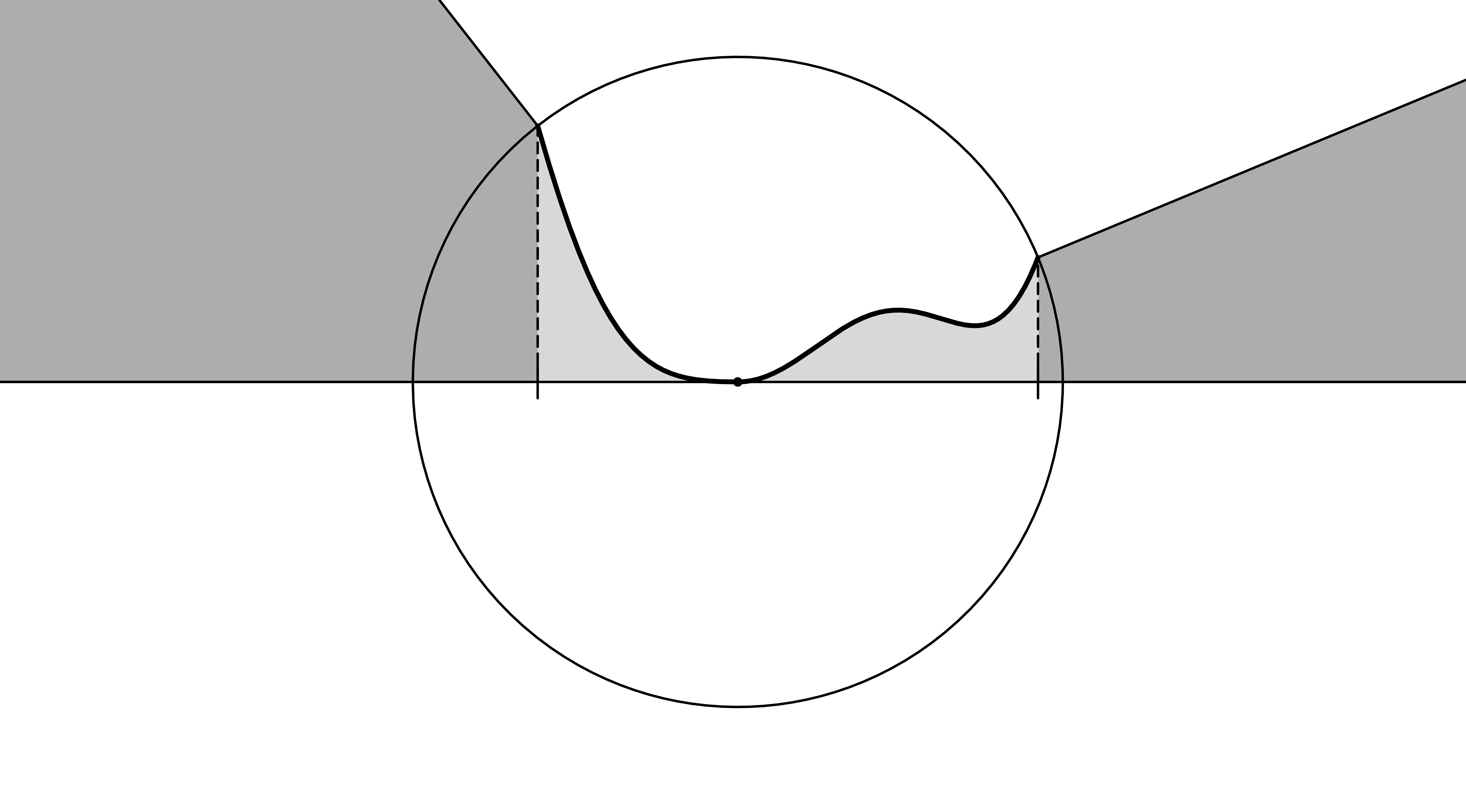}
    	\put(-199,79){$x_L$}
    	\put(-182,118){$\cC_\rho$}
    	\put(-155.5,82){$z$}
		\put(-98,80){$x_R$}
       	\caption{Depiction of $\cC_\rho$ along with $z$, where the leftmost dark gray region, the center light gray region, and the rightmost dark gray region denote the regions of the first, second, and third double integrals on the right-hand side of \eqref{threeint}, respectively.}
    	\label{fig:rho}
	\end{figure}
	Using $f_\rho$ and \eqref{ksrepi}, we obtain	
    \begin{equation}\label{threeint}
    	|\ks(z,\cC_\rho,\bu)|=\Big(\int_{-\infty}^{x_L}\int_{0}^{\tfrac{f_\rho(x_L)}{x_L}x}+\int_{x_L}^{x_R}\int_{0}^{f_\rho(x)}+\int_{x_R}^{\infty}\int_{0}^{\tfrac{f_\rho(x_R)}{x_R}x}\Big)(x^2+y^2)^{-(\sigma+2)/2}dydx.
    \end{equation}
    See Figure~\ref{fig:rho}. By the properties of $f_\rho$, Taylor's theorem implies that there exists a $C>0$ such that
    \begin{equation}\label{fbound_taylor}
    	|f_\rho(x)|\leq{}C|x|^{1+a},\qquad{}x\in{}[x_L,x_R].
    \end{equation}
    Examining the second term on the right-hand side of \eqref{threeint}, we have   
    \begin{align}
    	\int_{x_L}^{x_R}\int_{0}^{f_\rho(x)}(x^2+y^2)^{-(\sigma+2)/2}dydx &\leq{} \int_{x_L}^{x_R}f_\rho(x)|x|^{-\sigma-2}dx\\
	&\leq{} \int_{x_L}^{x_R}C|x|^{a-\sigma-1}dx\\
	&= \tfrac{C}{a-\sigma}\big(x_R^{a-\sigma}+(-x_L)^{a-\sigma}\big)\label{firstterm0}.
    \end{align}
    Similarly, we obtain
    \begin{align}
    	\int_{-\infty}^{x_L}\int_{0}^{\tfrac{f_\rho(x_L)}{x_L}x}(x^2+y^2)^{-(\sigma+2)/2}dydx &\leq{} \int_{-\infty}^{x_L}\tfrac{f_\rho(x_L)}{-x_L}|x|^{-\sigma-1}dx\\
	&= \frac{f_\rho(x_L)}{\sigma(-x_L)^{\sigma+1}}\leq{}\frac{C}{\sigma}(-x_L)^{a-\sigma}\label{f_rho_l_bound},
    \end{align}
    and
    \begin{align}
    	\int_{x_R}^{\infty}\int_{0}^{\tfrac{f_\rho(x_R)}{x_R}x}(x^2+y^2)^{-(\sigma+2)/2}dydx &\leq{} \int_{x_R}^{\infty}\tfrac{f_\rho(x_R)}{x_R}|x|^{-\sigma-1}dx\\
	&= \frac{f_\rho(x_R)}{\sigma x_R^{\sigma+1}} \leq{} \frac{C}{\sigma}x_R^{a-\sigma}\label{f_rho_r_bound}. 
    \end{align}
	Using \eqref{firstterm0}, \eqref{f_rho_l_bound}, and \eqref{f_rho_r_bound} in \eqref{threeint} yields
    \begin{align}
    	|\ks(z,\cC_\rho,\bu)| &\leq{} \tfrac{Ca}{\sigma(a-\sigma)}\big(x_R^{a-\sigma}+(-x_L)^{a-\sigma}\big)\\
	&\leq{}\tfrac{2Ca}{\sigma(a-\sigma)}\rho^{a-\sigma}.\label{final_lim_fin}
	\end{align}
    Thus, we may choose $\rho$ small enough so that
    \begin{equation}\label{bound_cc_rho}
    	|\ks(z,\cC_\rho,\bu)|<\tfrac{\ve}{3}.
    \end{equation}	
    Applying both this and \eqref{lt_eps_ov_3} to \eqref{finalbound_split_lr}, we obtain the desired result in this case.
    \vspace{0.1in}
	\hypertarget{step3thm1}{}
	
	\noindent{}\textbf{(Part 3)}
	Now, suppose that $\cC\subseteq{}\cP_+$ and $\cC$ is radial relative to $z$. We can write $\cC=\bigcup_{k\in\Nat}\cC_n$, where $\cC_n$, $n\in\Nat$, denote the connected components of $\cC$, with potentially some $\cC_k=\emptyset$. Following from Proposition~\ref{disj_thm}, we may find $N\in\Nat$ such that
	\begin{equation}\label{tail_lim_0}
		\big|\sum_{n=N+1}^{\infty}\ks(z,\cC_n,\bu)\big|<\frac{\ve}{2}.
	\end{equation}
	Then from Parts \hyperlink{step1thm1}{1} and \hyperlink{step2thm1}{2}, for every natural number $n\in[1,N]$, we can find a linear interpolating spline $\cI_n$ of $\cC_n$ such that
	\begin{equation}\label{jn_bound}
		|\ks(z,\cC_n,\bu)-\ks(z,\cI_n,\bu)|<\tfrac{\ve}{2N}.
	\end{equation}
	We then take $\widetilde{\cI}_N=\bigcup_{n=1}^{N}\cI_n$, and observe by Proposition~\ref{disj_thm} and the triangle inequality that
	\begin{align}
		|\ks(z,\cC,\bu)-\ks(z,\widetilde{\cI}_N,\bu)| &= \big|\sum_{n=1}^{\infty}\ks(z,\cC_n,\bu)-\sum_{n=1}^{N}\ks(z,\cI_n,\bu)\big|\\
		&\leq{}\sum_{n=1}^{N}\big|\ks(z,\cC_n,\bu)-\ks(z,\cI_n,\bu)\big|+\big|\sum_{n=N+1}^{\infty}\ks(z,\cC_n,\bu)\big|\label{s3_finisher2}.
	\end{align}
	Applying both \eqref{tail_lim_0} and \eqref{jn_bound} to \eqref{s3_finisher2} yields the desired result. A similar argument holds if $\cC\subseteq{}\cP_-$, and so the proof in this case will be omitted.
	    
    \vspace{0.1in}
	\hypertarget{step4thm1}{}
	\noindent{}\textbf{(Part 4)} 
	Now, suppose only that $\cC$ is radial relative to $z$. Let $\cC_+=\cC\cap{}\text{Int}(\cP_+)$ and $\cC_-=\cC\cap{}\text{Int}(\cP_-)$. Following from Part~\hyperlink{step3thm1}{3}, there are partial linear interpolating splines $\cI_+$ and $\cI_-$ of $\cC_+$ and $\cC_-$, respectively, such that
	\begin{equation}\label{pn_bound}
		|\ks(z,\cC_+,\bu)-\ks(z,\cI_+,\bu)|<\tfrac{\ve}{2}\quad\text{ and }\quad|\ks(z,\cC_-,\bu)-\ks(z,\cI_-,\bu)|<\tfrac{\ve}{2}.
	\end{equation}
	Taking $\cI$ to be the union of $\cI_+$ and $\cI_-$, we notice from Proposition~\ref{skpm} and the triangle inequality that
	\begin{equation}
		|\ks(z,\cC,\bu)-\ks(z,\cI,\bu)|\leq{}|\ks(z,\cC_+,\bu)-\ks(z,\cI_+,\bu)|+|\ks(z,\cC_-,\bu)-\ks(z,\cI_-,\bu)|.
	\end{equation}
	Combining this with \eqref{pn_bound}, we obtain the desired result.
	
	\vspace{0.1in}
	\hypertarget{step5thm1}{}
	\noindent{}\textbf{(Part 5)}
	Lastly, suppose that $\cC$ is any $C^{1,a}$ curve with $a>\sigma$. Let $\cC=\bigcup_{n\in\Nat}\cC_n$, where $\cC_n$ is defined as in \eqref{Cn}. Potentially $\cC_n$ is empty for large $n$. Following from Proposition~\ref{partit_thm} we may find some $N\in\Nat$ such that
	\begin{equation}\label{s5_finisher1}
		\big|\sum_{n=N+1}^{\infty}(-1)^{n+1}\ks(z,\cC_n,\bu)\big|<\frac{\ve}{2}.
	\end{equation}
	Then, from Part~\hyperlink{step4thm1}{4} it follows that for every natural number $n\in[1,N]$ there is a partial linear interpolating spline $\cI_n$ of $\cC_n$ such that
	\begin{equation}\label{s5_finisher2}
		|\ks(z,\cC_n,\bu)-\ks(z,\cI_n,\bu)|<\tfrac{\ve}{2N}.
	\end{equation}
	We then take $\widetilde{\cI}_n=\bigcup_{k=1}^{N}\cI_n$. From Proposition~\ref{partit_thm} and the triangle inequality we obtain
	\begin{align}
		|\ks(z,\cC,\bu)-\ks(z,\widetilde{\cI}_n,\bu)| &= \big|\sum_{n=1}^{\infty}(-1)^{n+1}\ks(z,\cC_n,\bu)-\sum_{n=1}^{N}(-1)^{n+1}\ks(z,\cI_n,\bu)\big|\\
		&\leq{}\sum_{n=1}^{N}|\ks(z,\cC_n,\bu)-\ks(z,\cI_n,\bu)|+\big|\sum_{n=N+1}^{\infty}\ks(z,\cC_n,\bu)\big|\label{s5_finisher}
	\end{align}
	Applying both \eqref{s5_finisher1} and \eqref{s5_finisher2} to \eqref{s5_finisher}, yields \eqref{lemma3state}.\qedwhite{}
\end{proof}

\section{Summary and discussion}\label{section:summary}

We now describe how the previous results can be combined to provide a method of approximating the nonlocal curvature of a $C^{1,a}$ curve $\cC$, with $a>\sigma$, at any point $z$ relative to a unit vector $\bu$. Namely, this is accomplished by going through the following steps.

\begin{itemize}
	\hypertarget{step1sec5}{}
	\item[\bf{(Step 1)}] If $z\in{}\cC$, choose a small $\rho>0$ and execute the following steps for the curve $\cC\setminus{}B_\rho(z)$.
	\hypertarget{step2sec5}{}
	\item[\bf{(Step 2)}] Use Proposition~\ref{partit_thm} to express the nonlocal curvature of the curve in terms of the nonlocal curvature of the curve's radial components $\cC_n$.
	\hypertarget{step3sec5}{}
	\item[\bf{(Step 3)}] Use Proposition~\ref{skpm} to write the nonlocal curvature of each radial component $\cC_n$ as the sum of the nonlocal curvature of $\cC_{n,+}=\cC_n\cap \text{\rm Int}(\cP_+)$ and $\cC_{n,-}=\cC_n\cap \text{\rm Int}(\cP_+)$.
	\hypertarget{step4sec5}{}
	\item[\bf{(Step 4)}] Use Proposition~\ref{disj_thm} to write the nonlocal curvature of each $\cC_{n,+}$ and $\cC_{n,-}$ as the sum of the nonlocal curvatures of their respective disjoint components $\cC_{n,+}^i$ and $\cC_{n,-}^j$, where $i$ and $j$ range the number of connected components of $\cC_{n,+}$ and $\cC_{n,-}$, respectively.
	\hypertarget{step5sec5}{}
	\item[\bf{(Step 5)}] Choose a suitably large $N\in\Nat$ and using the argument in Part~\hyperlink{step1thm1}{1} of Theorem~\ref{lemma3}, construct $N$-point linear interpolating splines $\cS_{n,+}^{i}$ and $\cS_{n,-}^{j}$ for every $\cC_{n,+}^{i}$ and $\cC_{n,-}^{j}$, respectively.
	\hypertarget{step6sec5}{}
	\item[\bf{(Step 6)}] Making use of Proposition~\ref{disj_thm} and Lemma~\ref{nlc_2p_ldef}, calculate the nonlocal curvature of each $\cS_{n,+}^{i}$ and $\cS_{n,-}^{j}$ for every $\cC_{n,+}^{i}$ and $\cC_{n,-}^{j}$, respectively.
\end{itemize}

Theorem~\ref{lemma3} ensures that the above procedure allows for the approximation of the nonlocal curvature of $\cC$ relative to $z$ to any degree of closeness. To do so, one may increase the number of interpolated points $N$ mentioned in Step 5 when constructing the linear interpolating splines, and additionally decrease $\rho$ mentioned in Step 1 if $z\in{}\cC$.

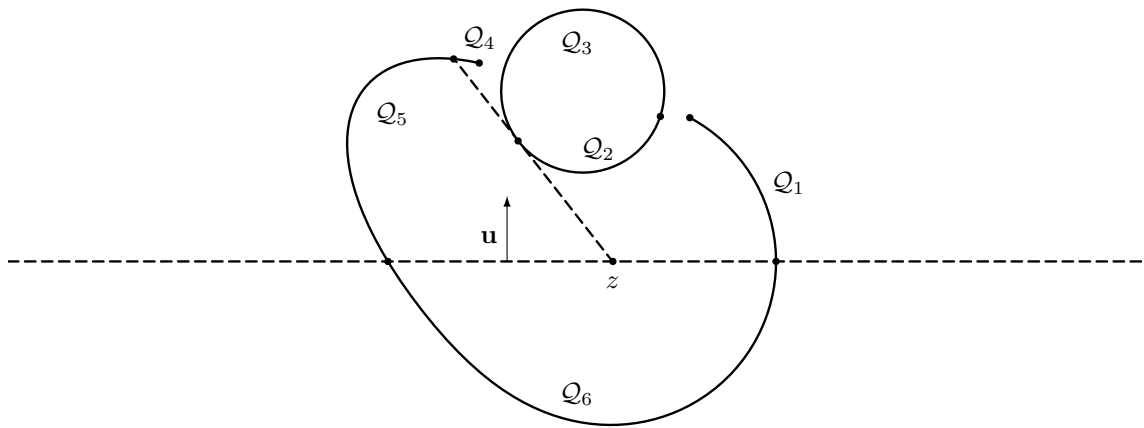
\begin{figure}[h]
    \centering%
    \includegraphics[height=5.6cm]{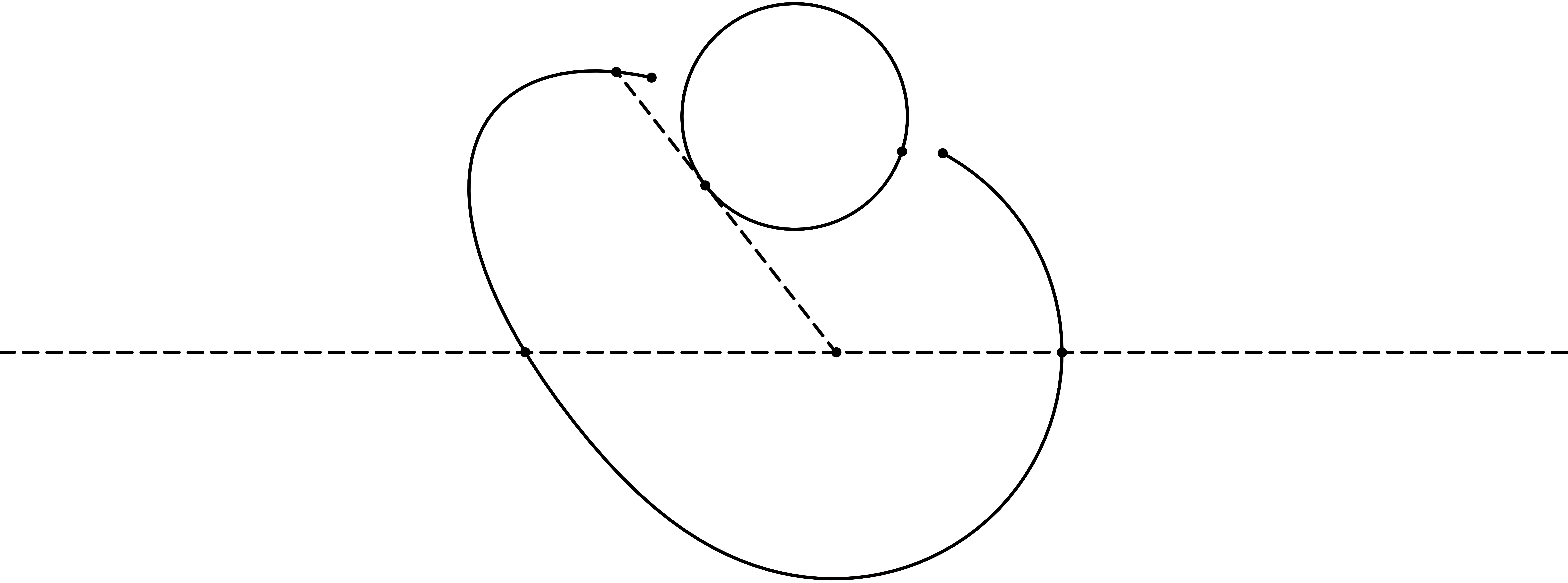}
    \put(-202.6,53){$z$}
    \put(-140,90){$\cQ_1$}
    \put(-212,103){$\cQ_2$}
    \put(-220,143){$\cQ_3$}
    \put(-257,145){$\cQ_4$}
    \put(-290,116){$\cQ_5$}
    \put(-220,10){$\cQ_6$}
    \put(-240,63){$\vector(0,1){25}$}
    \put(-250,70){$\bu$}
    \caption{Depiction of a curve $\cC$ that is decomposed into six pieces. A point $z$ and unit vector $\bu$ are also shown. The dashed lines illustrate how the decomposition of the curve is relate to the point $z$ and the unit vector $\bu$.}
	\label{fig:crazy}
\end{figure}

We will illustrate how the above steps can be carried out by considering an example. Consider the curve $\cC$ depicted in Figure~\ref{fig:crazy} consisting of disjoint components $\cQ_i$, for $i=1,\dots{},6$. Furthermore, consider the $z\in{}\Real^2$ and unit vector $\bu$ show in Figure~\ref{fig:crazy}. We now follow the previously mentioned steps.

\begin{itemize}
	\item[\bf{(Step 1)}] As $z\notin{}\cC$, we can skip this step.
	
	\item[\bf{(Step 2)}] In this step, we first decompose $\cC$ into radial components. Towards this end, we write $\cC=\cC_1\cup\cC_2\cup\cC_3$, where $\cC_1=\cQ_1\cup\cQ_2\cup\cQ_5\cup\cQ_6$, $\cC_2=\cQ_3$, and $\cC_3=\cQ_4$. According to Proposition~\ref{partit_thm}, we then have
	\begin{equation}\label{step2eqn}
		\ks(z,\cC,u)=\ks(z,\cQ_1\cup\cQ_2\cup\cQ_5\cup\cQ_6,\bu)-\ks(z,\cQ_3,\bu)+\ks(z,\cQ_4,\bu).
	\end{equation}
	\item[\bf{(Step 3)}] Next, we decompose the curves from the previous step into their components that lie in each half plane determined by $z$ and $\bu$. We have $\cC_1=\cC_{1,+}\cup{}\cC_{1,-}$, where $\cC_{1,+}=\cQ_1\cup{}\cQ_2\cup{}\cQ_5$ and $\cC_{1,-}=\cQ_6$, $\cC_2=\cC_{2,+}\cup{}\cC_{2,-}$, where $\cC_{2,+}=\cQ_3$ and $\cC_{2,-}=\emptyset$, and lastly $\cC_3=\cC_{3,+}\cup{}\cC_{3,-}$ where $\cC_{3,+}=\cQ_4$ and $\cC_{3,-}=\emptyset$. Therefore, Proposition~\ref{skpm} and \eqref{step2eqn} yield
	\begin{equation}\label{step3eqn}
		\ks(z,\cC,\bu)=\ks(z,\cQ_1\cup{}\cQ_2\cup{}\cQ_5,\bu)+\ks(z,\cQ_6,\bu)-\ks(z,\cQ_3,\bu)+\ks(z,\cQ_4,\bu).
	\end{equation}
	\item[\bf{(Step 4)}] We now seek to decompose the curves introduced in the prior step into their disjoint components. The only one of these curves that is not connected is $\cC_{1,+}$ which can be written as the disjoint union of $\cQ_1$, $\cQ_2$, and $\cQ_5$. Hence, by Proposition~\ref{disj_thm} and \eqref{step3eqn}, we have
	\begin{equation*}
		\ks(z,\cC,\bu)=\ks(z,\cQ_1,\bu)+\ks(z,\cQ_2,\bu)+\ks(z,\cQ_4,\bu)+\ks(z,\cQ_5,\bu)+\ks(z,\cQ_6,\bu)-\ks(z,\cQ_3,\bu).
	\end{equation*}
	\item[\bf{(Step 5)}] As in Part~\hyperlink{step1thm1}{1} of Theorem~\ref{lemma3}, we consider a large $N\in\Nat$ and take $\ell_{i,N}$ to be a partial linear interpolating spline with $N$ points of $\cQ_i$ for $i=1,\dots{},6$. We then find
	\begin{align}
		\nonumber{} \ks(z,\cC,\bu) &\approx{} \ks(z,\ell_{1,N},\bu)+\ks(z,\ell_{2,N},\bu)+\ks(z,\ell_{4,N},\bu)\\
		\label{Step5Final}&+\ks(z,\ell_{5,N},\bu)+\ks(z,\ell_{6,N},\bu)-\ks(z,\ell_{3,N},\bu).
	\end{align}
	\item[\bf{(Step 6)}] Making use of Proposition~\ref{disj_thm} we can express the nonlocal curvature of each $\ell_{i,N}$ in terms of the nonlocal curvatures of line segments, and then use Lemma~\ref{nlc_2p_ldef} to evaluate the nonlocal curvature of the segments. Substituting this information back into \eqref{Step5Final}, we obtain the final approximation for the nonlocal curvature of $\cC$.
\end{itemize}

\providecommand{\bysame}{\leavevmode\hbox to3em{\hrulefill}\thinspace}
\providecommand{\MR}{\relax\ifhmode\unskip\space\fi MR }
\providecommand{\MRhref}[2]{%
  \href{http://www.ams.org/mathscinet-getitem?mr=#1}{#2}
}
\providecommand{\href}[2]{#2}

\bibliographystyle{amsplain}

\end{document}